\documentclass{llncs}
\usepackage[cp1250]{inputenc}
\usepackage[english]{babel}
\usepackage{amsmath,amssymb,amsfonts,euscript}
\usepackage{tikz}
\usepackage{mathtools}
\usepackage{bussproofs}

\newtheorem{thm}{Theorem}[section]
\newtheorem{prop}[thm]{Proposition}
\newtheorem{lem}[thm]{Lemma}

\begin{document}

\author{Yury Savateev 
\inst{1}
\and
Daniyar Shamkanov
\inst{1}\inst{2}
}
\institute{National Research University Higher School of Economics
\and
Steklov Mathematical Institute of the Russian Academy of Sciences}

\title{Cut-elimination for the modal Grzegorczyk logic via non-well-founded proofs}
\date{}


\maketitle

\begin{abstract}
We present a sequent calculus for the modal Grzegorczyk logic $\mathsf{Grz}$ allowing non-well-founded proofs and obtain the cut-elimination theorem for it by constructing a continuous cut-elimination mapping acting on these proofs. 
\end{abstract}

\section{Introduction}
\label{s1}
The Grzegorczyk logic $\mathsf{Grz}$ is a well-known modal logic \cite{Maks}, which can be characterized by reflexive partially ordered Kripke frames without infinite ascending chains. This logic is complete w.r.t. the arithmetical semantics, where the modal connective $\Box$ corresponds to the strong provability operator \textit{"... is true and provable"} in Peano arithmetic.

Recently a new proof-theoretic description for the G\"{o}del-L\"{o}b provability logic $\mathsf{GL}$ in the form of a sequent calculus  allowing so-called cyclic, or circular, proofs was given in \cite{Sham}. 
A feature of cyclic proofs is that the graph underlying a proof is not a finite tree but is allowed to contain cycles. Since $\mathsf{GL}$ and $\mathsf{Grz}$ are closely connected, we wonder whether cyclic and, more generally, non-well-founded proofs can be fruitfully considered in the case of $\mathsf{Grz}$. 

In this paper, we present a sequent calculus for the modal Grzegorczyk logic allowing non-well-founded proofs and obtain the cut-elimination theorem for it by constructing a continuous cut-elimination mapping acting on these proofs. 
  
In Section \ref{s2}, we recall an ordinary sequent calculus for $\mathsf{Grz}$. In Section \ref{s3} we introduce the infinitary proof system $\mathsf{Grz_{\infty}}$. In Section \ref{s4} we establish the cut elimination result for $\mathsf{Grz_\infty}$ syntactically. Then, in Section \ref{s5} we prove the equivalence of the two systems. In Section \ref{s6} we discuss possible applications of the new system.

\section{Preliminaries}
\label{s2}
In this section we recall the modal Grzegorczyk logic $\mathsf{Grz}$ and define an ordinary sequent calculus for it.
  
\textit{Formulas} of $\mathsf{Grz}$, denoted by $A$, $B$, $C$, are built up as follows:
$$ A ::= \bot \,\,|\,\, p \,\,|\,\, (A \to A) \,\,|\,\, \Box A \;, $$
where $p$ stands for atomic propositions. 
We treat other boolean connectives and the modal operator $\Diamond$ as abbreviations:
\begin{gather*}
\neg A := A\to \bot,\qquad\top := \neg \bot,\qquad A\wedge B := \neg (A\to \neg B),
\\
A\vee B := (\neg A\to B),\qquad\Diamond A := \neg\Box \neg A.
\end{gather*}

The Hilbert-style axiomatization of $\mathsf{Grz}$ is given by the following axioms and inference rules:

\textit{Axioms:}
\begin{itemize}
\item[(i)] Boolean tautologies;
\item[(ii)] $\Box (A \rightarrow B) \rightarrow (\Box A \rightarrow \Box B)$;
\item[(iii)] $\Box A \rightarrow \Box \Box A$;
\item[(iv)] $\Box A \rightarrow A$;
\item[(v)] $\Box(\Box(A \rightarrow \Box A) \rightarrow A) \rightarrow \Box A$.
\end{itemize}

\textit{Rules:} modus ponens, $A / \Box A$. \\

Now we define an ordinary sequent calculus for $\mathsf{Grz}$. A \textit{sequent} is an expression of the form $\Gamma \Rightarrow \Delta$, where $\Gamma$ and~$\Delta$ are finite multisets of formulas. For a multiset of formulas $\Gamma = A_1,\dotsc, A_n$, we set $\Box \Gamma := \Box A_1,\dotsc, \Box A_n$.

The system $\mathsf{Grz_{Seq}}$, which is a variant of the sequent calculus from \cite{Borga}, is defined by the following initial sequents and inference rules: 
\begin{gather*}
\AxiomC{ $\Gamma, A \Rightarrow A, \Delta $ ,}
\DisplayProof \qquad
\AxiomC{ $\Gamma , \bot \Rightarrow \Delta $ ,}
\DisplayProof
\end{gather*}
\begin{align*}
&
\AxiomC{$\Gamma , B \Rightarrow \Delta $}
\AxiomC{$\Gamma \Rightarrow A,\Delta $}
\LeftLabel{$\mathsf{\to_L}$}
\BinaryInfC{$\Gamma , A \to B \Rightarrow \Delta$}
\DisplayProof\;,& &
\AxiomC{$\Gamma , A \Rightarrow B, \Delta $}
\LeftLabel{$\mathsf{\to_R}$}
\UnaryInfC{$\Gamma \Rightarrow A \to B, \Delta$}
\DisplayProof\;,\\\\
&\AxiomC{$\Gamma, B, \Box B \Rightarrow \Delta $}
\LeftLabel{$\mathsf{refl}$}
\UnaryInfC{$\Gamma , \Box B \Rightarrow \Delta$}
\DisplayProof\;,& &
\AxiomC{$ \Box\Pi, \Box(A\to\Box A) \Rightarrow A$}
\LeftLabel{$\mathsf{\Box_{Grz}}$}
\UnaryInfC{$\Gamma, \Box \Pi \Rightarrow \Box A ,\Delta $}
\DisplayProof \;.
\end{align*}
\begin{center}
\textbf{Fig. 1.} The system $\mathsf{Grz_{Seq}}$
\end{center}

The cut rule has the form
\begin{gather*}
\AxiomC{$\Gamma\Rightarrow A,\Delta$}
\AxiomC{$\Gamma,A\Rightarrow\Delta$}
\LeftLabel{$\mathsf{cut}$}
\RightLabel{ ,}
\BinaryInfC{$\Gamma\Rightarrow\Delta$}
\DisplayProof
\end{gather*}
where $A$ is called the \emph{cut formula} of the given inference.

\begin{lem} \label{prop}
$\mathsf{Grz_{Seq}} + \mathsf{cut}\vdash \Gamma\Rightarrow\Delta$ if and only if $\mathsf{Grz} \vdash \bigwedge\Gamma\to\bigvee\Delta $. 
\end{lem}
\begin{proof}
Standard transformations of proofs.
\end{proof}

\begin{thm}\label{cutelimgrz}
If $\mathsf{Grz_{Seq}} + \mathsf{cut}\vdash \Gamma\Rightarrow\Delta$, then $\mathsf{Grz_{Seq}} \vdash \Gamma\Rightarrow\Delta$. 

\end{thm}

A syntactic cut-elimination for the logic $\mathsf{Grz}$ was obtained by M.~Borga and P.~Gentilini in \cite{Borga}. 
In this paper, we will give another proof of this cut-elimination theorem in the next sections.

\section{Non-well-founded proofs}
\label{s3}
Now we define a sequent calculus for the logic $\mathsf{Grz}$ allowing non-well-founded proofs. The cut-elimination theorem for it will be proved in the next section. 

Inference rules and initial sequents of the sequent calculus $\mathsf{Grz_\infty}$ have the following form:
\begin{gather*}
\AxiomC{ $\Gamma, p \Rightarrow p, \Delta $ ,}
\DisplayProof\qquad
\AxiomC{ $\Gamma , \bot \Rightarrow  \Delta$ ,}
\DisplayProof \\
\AxiomC{$\Gamma , B \Rightarrow  \Delta$}
\AxiomC{$\Gamma \Rightarrow  A, \Delta$}
\LeftLabel{$\mathsf{\rightarrow_L}$}
\BinaryInfC{$\Gamma , A \rightarrow B \Rightarrow  \Delta$}
\DisplayProof \;,\qquad
\AxiomC{$\Gamma, A \Rightarrow  B ,\Delta$}
\LeftLabel{$\mathsf{\rightarrow_R}$}
\UnaryInfC{$\Gamma \Rightarrow  A \rightarrow B ,\Delta$}
\DisplayProof \;,\\
\AxiomC{$\Gamma, A, \Box A \Rightarrow \Delta$}
\LeftLabel{$\mathsf{refl}$}
\UnaryInfC{$\Gamma ,\Box A \Rightarrow \Delta$}
\DisplayProof \;,\qquad
\AxiomC{$\Gamma, \Box \Pi \Rightarrow A, \Delta$}
\AxiomC{$\Box \Pi \Rightarrow A$}
\LeftLabel{$\mathsf{\Box}$}
\BinaryInfC{$\Gamma, \Box \Pi \Rightarrow \Box A, \Delta$}
\DisplayProof \;.
\end{gather*}
\begin{center}
\textbf{Fig. 2.} The system $\mathsf{Grz}_\infty$
\end{center}
The system $\mathsf{Grz}_{\infty}+\mathsf{cut}$ is defined by adding the rule ($\mathsf{cut}$) to the system $\mathsf{Grz_\infty}$.
An \emph{$\infty$--proof} in $\mathsf{Grz}_\infty$ ($\mathsf{Grz}_{\infty}+\mathsf{cut}$) is a (possibly infinite) tree whose nodes are marked by
sequents and whose leaves are marked by initial sequents and that is constructed according to the rules of the sequent calculus. In addition, every infinite branch in an $\infty$--proof must pass through a right premise of the rule $\Box$ infinitely many times. A sequent $\Gamma \Rightarrow \Delta$ is \emph{provable} in $\mathsf{Grz}_\infty$ ($\mathsf{Grz}_{\infty}+\mathsf{cut}$) if there is an $\infty$--proof in $\mathsf{Grz}_\infty$ ($\mathsf{Grz}_{\infty}+\mathsf{cut}$) with the root marked by $\Gamma \Rightarrow \Delta$.

The \emph{main fragment} of an $\infty$--proof is a finite tree obtained from the $\infty$--proof by cutting every infinite branch at the nearest to the root right premise of the rule ($\Box$).
The \emph{local height $\lvert \pi \rvert$ of an $\infty$--proof $\pi$} is the length of the longest branch in its main fragment. An $\infty$--proof only consisting of an initial sequent has height 0.

For instance, consider an $\infty$--proof of the sequent $\Box(\Box(p \rightarrow \Box p) \rightarrow p) \Rightarrow p$: 

\begin{gather*}
\AxiomC{\textsf{Ax}}
\noLine
\UnaryInfC{$ F, p\Rightarrow p$}
\AxiomC{\textsf{Ax}}
\noLine
\UnaryInfC{$ F,p\Rightarrow \Box p,p$}
\LeftLabel{$\mathsf{\to_R}$}
\UnaryInfC{$ F \Rightarrow p\to\Box p,p$}
\AxiomC{\textsf{Ax}}
\noLine
\UnaryInfC{$p, F \Rightarrow p$}
\AxiomC{$\vdots$}
\noLine
\UnaryInfC{$ F \Rightarrow p$} 
\LeftLabel{$\mathsf{\Box}$} 
\BinaryInfC{$p, F \Rightarrow \Box p$}
\LeftLabel{$\mathsf{}\to_R$}
\UnaryInfC{$ F \Rightarrow p\to\Box p$}
\LeftLabel{$\mathsf{\Box}$} 
\BinaryInfC{$ F \Rightarrow \Box(p\to \Box p),p$} 
\LeftLabel{$\mathsf{\to_L}$}
\BinaryInfC{$\Box(p \rightarrow \Box p) \rightarrow p, F \Rightarrow p$}
\LeftLabel{$\mathsf{refl}$}
\RightLabel{ ,} 
\UnaryInfC{$F \Rightarrow p$}
\DisplayProof
\end{gather*}
where $F=\Box(\Box(p \rightarrow \Box p) \rightarrow p) $. 
The local height of this $\infty$--proof equals to 4 and its main fragment has the form
\begin{gather*}
\AxiomC{\textsf{Ax}}
\noLine
\UnaryInfC{$ F, p\Rightarrow p$}
\AxiomC{\textsf{Ax}}
\noLine
\UnaryInfC{$ F,p\Rightarrow \Box p,p$}
\LeftLabel{$\mathsf{\to_R}$}
\UnaryInfC{$ F \Rightarrow p\to\Box p,p$}
\AxiomC{\qquad \qquad \qquad}
\LeftLabel{$\mathsf{\Box}$} 
\BinaryInfC{$ F \Rightarrow \Box(p\to \Box p),p$} 
\LeftLabel{$\mathsf{\to_L}$}
\BinaryInfC{$\Box(p \rightarrow \Box p) \rightarrow p, F \Rightarrow p$}
\LeftLabel{$\mathsf{refl}$}
\RightLabel{ .} 
\UnaryInfC{$F \Rightarrow p $}
\DisplayProof
\end{gather*}

By $\mathcal P$ denote the set of all $\infty$-proofs in $\mathsf{Grz}_{\infty} +\mathsf{cut} $. For $n \in \mathbb{N}$, we define binary relations $\sim_n$ on $\mathcal P$ by simultaneous induction:
\begin{enumerate}
\item $\pi \sim_0 \tau$ for any $\pi, \tau$;
\item if $\lvert \pi \rvert =0$, then $\pi \sim_n \pi$;
\item if $\pi$ and $\tau$ are obtained by the same instance of inference rules ($\mathsf{\to_L}$), ($\mathsf{cut}$) from $\pi^\prime$, $\pi^{\prime\prime}$ and $\tau^\prime$, $\tau^{\prime\prime}$, where $\pi^\prime \sim_n \tau^\prime$ and $\pi^{\prime\prime} \sim_n \tau^{\prime\prime}$, then $\pi \sim_{n} \tau$;
\item if $\pi$ and $\tau$ are obtained by the same instance of inference rules ($\mathsf{\to_R}$), ($\mathsf{refl}$) from $\pi^\prime$ and $\tau^\prime$, where $\pi^\prime \sim_n \tau^\prime$, then $\pi \sim_{n} \tau$; 
\item if $\pi$ and $\tau$ are obtained by the same instance of an inference rule ($\mathsf{\Box}$) from $\pi^\prime$, $\pi^{\prime\prime}$ and $\tau^\prime$, $\tau^{\prime\prime}$, where $\pi^\prime$, $\tau^\prime$ are $\infty$-proofs for the left premises of ($\mathsf{\Box}$), and $\pi^\prime \sim_{n+1} \tau^\prime$, $\pi^{\prime\prime} \sim_n \tau^{\prime\prime}$, then $\pi \sim_{n+1} \tau$. 
\end{enumerate}
Notice that $\pi \sim_1 \tau$ if and only if $\pi$ and $\tau$ have the same main fragment.
\begin{lem}
For any $n \in \mathbb{N}$, we have that
\begin{enumerate}
\item the relation $\sim_n$ is an equivalence relation;
\item the relation $\sim_{n+1}$ is finer than the relation $\sim_n$.
\end{enumerate}
In addition, the intersection of all relations $ \sim_n $ is exactly the equality relation over $\mathcal P$.
\end{lem}

Now we define a sequence $\mathcal{P}_n$ of subsets of $\mathcal{P}$ by simultaneous induction:
\begin{enumerate}
\item $\pi \in \mathcal{P}_0$ for any $\pi$;
\item if $\lvert \pi \rvert =0$, then $\pi \in \mathcal{P}_n$;
\item if $\pi$ is obtained by an instance of an inference rule ($\mathsf{\to_L}$) from $\pi^\prime$ and $\pi^{\prime\prime}$, where $\pi^\prime, \pi^{\prime\prime} \in \mathcal{P}_n$, then $\pi \in \mathcal{P}_n$;
\item if $\pi$ is obtained by an instance of inference rules ($\mathsf{\to_R}$), ($\mathsf{refl}$) from $\pi^\prime$, where $\pi^\prime \in \mathcal{P}_n$, then $\pi \in \mathcal{P}_n$; 
\item if $\pi$ is obtained by an instance of an inference rule ($\mathsf{\Box}$) from $\pi^\prime$ and $\pi^{\prime\prime}$, where $\pi^\prime$ is an $\infty$-proof for the left premise of ($\mathsf{\Box}$), and $\pi^\prime \in \mathcal{P}_{n+1}$, $\pi^{\prime\prime} \in \mathcal{P}_n$, then $\pi \in \mathcal{P}_{n+1}$. 
\end{enumerate}
Notice that $\mathcal{P}_0= \mathcal{P}$ and $\mathcal{P}_1$ consists of the $\infty$-proofs that do not contain the cut rule in their main fragment.
\begin{lem}
We have that $\mathcal{P}_n \subset \mathcal{P}_{n+1}$ for any $n \in \mathbb{N}$. In addition, the intersection of all sets $ \mathcal{P}_n $ consists exactly of the $\infty$-proofs in $\mathsf{Grz}_{\infty}$. 
\end{lem}
For $\pi, \tau \in \mathcal{P}$, we define $d(\pi,\tau) = 2^{- \sup \{n \in \mathbb{N} \:\mid\: \pi \sim_n \tau\}}$, where by convention $2^{-\infty} =0$. We see that an equivalence $\pi \sim_n \tau$ holds if and only if $d(\pi , \tau) \leqslant 2^{-n}$. 


\begin{prop}\label{ComplP}
$(\mathcal{P},d)$ is a complete metric space.
\end{prop} 

A mapping $\mathcal{U} \colon \mathcal{P}^k_m \to \mathcal{P}_m$ is \emph{nonexpansive} if for any $n \in \mathbb{N}$
$$\pi_1 \sim_n \tau_1, \dotsc , \pi_k \sim_n \tau_k \Rightarrow \mathcal{U}(\pi_1, \dotsc, \pi_k) \sim_n \mathcal{U}(\tau_1, \dotsc, \tau_k)\:,$$
which is equivalent to the standard condition $$d(\mathcal{U}(\pi_1, \dotsc, \pi_k),\mathcal{U}(\tau_1, \dotsc, \tau_k)) \leqslant \max\{d(\pi_1,\tau_1) ,\dotsc, d(\pi_k,\tau_k) \}\:.$$
Trivially, any nonexpansive mapping is continuous.

A nonexpansive mapping $\mathcal{U} \colon \mathcal{P} \to \mathcal{P}$ is called \emph{adequate} if 
$ \mathcal{U}(\mathcal{P}_1)\subset \mathcal{P}_1$ and $\lvert \mathcal{U}(\pi)\rvert \leqslant \lvert \pi \rvert$ for any $\pi \in \mathcal{P}$.

Recall that an inference rule is called \emph{admissible} (in a given proof system) if, for any instance of the rule, the conclusion is provable whenever all premises are provable. 
In $\mathsf{Grz}_{\infty} +\mathsf{cut}$, we call a single-premise inference rule \emph{strongly admissible} if there is an adequate mapping $\mathcal{U}\colon\mathcal{P} \to \mathcal{P}$ that maps any $\infty$-proof of the premise of the rule to an $\infty$-proof of the conclusion.

\begin{lem}\label{strongweakening}
For any finite multisets of formulas $\Pi$ and $\Sigma$, the inference rule
\begin{gather*}
\AxiomC{$\Gamma\Rightarrow\Delta$}
\LeftLabel{$\mathsf{wk}_{\Pi, \Sigma}$}
\UnaryInfC{$\Pi,\Gamma\Rightarrow\Delta,\Sigma$}
\DisplayProof
\end{gather*}
is strongly admissible in $\mathsf{Grz}_{\infty} +\mathsf{cut}$. 
\end{lem}

\begin{lem}\label{inversion}
For any formulas $A$ and $B$, the rules
\begin{gather*}
\AxiomC{$\Gamma , A \rightarrow B \Rightarrow  \Delta$}
\LeftLabel{$\mathsf{li}_{A \to B}$}
\UnaryInfC{$\Gamma ,B \Rightarrow  \Delta$}
\DisplayProof\qquad
\AxiomC{$\Gamma , A \rightarrow B \Rightarrow  \Delta$}
\LeftLabel{$\mathsf{ri}_{A \to B}$}
\UnaryInfC{$\Gamma  \Rightarrow  A, \Delta$}
\DisplayProof\\\\
\AxiomC{$\Gamma  \Rightarrow   A \rightarrow B, \Delta$}
\LeftLabel{$\mathsf{i}_{A \to B}$}
\UnaryInfC{$\Gamma ,A \Rightarrow B, \Delta$}
\DisplayProof\qquad
\AxiomC{$\Gamma  \Rightarrow   \bot, \Delta$}
\LeftLabel{$\mathsf{i}_{\bot}$}
\UnaryInfC{$\Gamma  \Rightarrow \Delta$}
\DisplayProof\qquad 
\AxiomC{$\Gamma \Rightarrow \Box A, \Delta$}
\LeftLabel{$\mathsf{li}_{\:\Box A}$}
\UnaryInfC{$\Gamma \Rightarrow A, \Delta$}
\DisplayProof
\end{gather*}
are strongly admissible in $\mathsf{Grz}_{\infty} +\mathsf{cut}$.
\end{lem}

\begin{lem}\label{weakcontraction}
For any atomic proposition $p$, the rules
\begin{gather*}
\AxiomC{$\Gamma , p,p \Rightarrow  \Delta$}
\LeftLabel{$\mathsf{acl}_{p}$}
\UnaryInfC{$\Gamma ,p \Rightarrow  \Delta$}
\DisplayProof\qquad
\AxiomC{$\Gamma \Rightarrow p,p, \Delta$}
\LeftLabel{$\mathsf{acr}_{p}$}
\UnaryInfC{$\Gamma \Rightarrow p, \Delta$}
\DisplayProof
\end{gather*}
are strongly admissible in $\mathsf{Grz}_{\infty} +\mathsf{cut}$.
\end{lem}

These lemmata can be obtained in a standard way, so we omit the proofs.

\section{Cut elimination}
\label{s4}

In this section we construct a continuous cut elimination mapping from $\mathcal{P}$ to $\mathcal{P}$, which eliminates all applications of the cut rule from any $\infty$-proof in $\mathsf{Grz}_{\infty} +\mathsf{cut}$. In what follows, we use nonexpansive mappings $\mathcal \mathit{wk}_{\Pi, \Sigma}$, $\mathit{li}_{A\to B}$, $\mathit{ri}_{A\to B}$, $\mathit{i}_{A\to B}$, $\mathit{i}_\bot$, $\mathit{li}_{\Box A}$, $\mathit{acl}_{p}$, $\mathit{acr}_{p}$ from Lemma \ref{strongweakening}, Lemma \ref{inversion} and Lemma \ref{weakcontraction}.

For a modal formula $A$, a nonexpansive mapping $\mathcal{R}$ from $\mathcal P_1 \times \mathcal P_1$ to $\mathcal P_1$ is called
\emph{$A$-reducing} if $\mathcal R(\pi^\prime,\pi^{\prime\prime})$ is an $\infty$-proof of $\Gamma\Rightarrow \Delta$ whenever $\pi^\prime$ is an $\infty$-proof of $\Gamma\Rightarrow \Delta, A$ and $\pi^{\prime\prime}$ is an $\infty$-proof of $A, \Gamma\Rightarrow \Delta$.

\begin{lem}\label{smallcut}
For any atomic proposition $p$ there is a $p$-reducing mapping $\mathcal{R}_p$.
\end{lem}

\begin{lem}\label{boxcut}
Given a $B$-reducing mapping $\mathcal{R}_B$, there is a $\Box B$-reducing mapping $\mathcal{R}_{\Box B}$.
\end{lem}
The proof of these two Lemmas can be found in the Appendix.

\begin{lem}
For any formula $A$, there is an $A$-reducing mapping $\mathcal{R}_A$.
\end{lem}
\begin{proof}
We define $\mathcal{R}_{A}$ by induction on the structure of the formula $A $.

Case 1: $A$ has the form $p$. In this case, $\mathcal{R}_{p} $ is defined in Lemma \ref{smallcut}.
 
Case 2: $A$ has the form $\bot$. Then we put $\mathcal{R}_{\bot}(\pi^\prime,\pi^{\prime\prime}):= \mathit{i}_\bot (\pi^\prime) $, where $\mathit{i}_\bot$ is a nonexpansive mapping from Lemma \ref{inversion}.

Case 3: $A$ has the form $B\to C$. Then we put    
$$\mathcal R_{B \to C}(\pi^\prime,\pi^{\prime\prime}):=\mathcal R_C(\mathcal R_B(\mathcal \mathit{wk}_{\emptyset, C}(\mathit{ri}_{B\to C}(\pi^{\prime\prime})),\mathit{i}_{B\to C}(\pi^\prime)),\mathit{li}_{B\to C}(\pi^{\prime\prime}))\;,$$
where $\mathit{ri}_{B \to C}$, $\mathit{i}_{B \to C}$, $\mathit{li}_{B \to C}$ are nonexpansive mappings from Lemma \ref{inversion} and $\mathit{wk}_{\emptyset, C}$ is a nonexpansive mapping from Lemma \ref{strongweakening}.

Case 4: $A$ has the form $\Box B$. By the induction hypothesis, there is a $B$-reducing mapping $\mathcal{R}_B$. By Lemma \ref{boxcut} there is a $\Box B$-reducing mapping $\mathcal{R}_{\Box B}$. 

\end{proof}

A mapping $\mathcal{U}\colon \mathcal P \to\mathcal P$ is called \emph{root-preserving} if it maps $\infty$-proofs to $\infty$-proofs of the same sequents.
The set of all root-preserving nonexpansive mappings from $\mathcal P$ to $\mathcal P$ is denoted by $\mathcal{N}$. 
We consider $ \mathcal{N}$ as a metric space with the uniform metric:
$$d(\mathcal{U},\mathcal{V})=\sup_{\pi\in \mathcal P}d(\mathcal{U}(\pi),\mathcal{V}(\pi))\:.$$

\begin{lem}
$(\mathcal{N}, d)$ is a non-empty complete metric space.
\end{lem}
\begin{proof}
By Lemma \ref{ComplP}, $\mathcal{P}$ is a complete metric space. Consequently the set $\mathit{C}(\mathcal{P},\mathcal{P})$ of all continuous mappings from $\mathcal{P}$ to $\mathcal{P}$ with the uniform metric forms a complete metric space. The reader will easily prove that $\mathcal{N}$ is a closed subset of $\mathit{C}(\mathcal{P},\mathcal{P})$. In addition, the set $\mathcal{N}$ is non-empty, because the identity mapping belongs to $\mathcal{N}$. Thus $(\mathcal{N}, d)$ is a non-empty complete metric space.     
\end{proof}

We define $\mathcal{N}_n:= \{\mathcal{U}\in \mathcal{N}\: \mid \: \mathcal{U}(\mathcal{P}) \subset \mathcal{P}_n\}$.

\begin{lem}\label{main-fragment cut}
There exists a mapping $\mathcal{E}^\ast \in \mathcal{N}_1$.
\end{lem}
\begin{proof}
Assume we have an $\infty$-proof $\pi$. We define $\mathcal{E}^\ast(\pi)$ by induction on $\lvert\pi\rvert$.

If $\lvert\pi\rvert=0$, then we put $\mathcal{E}^\ast(\pi)=\pi$. Otherwise, consider the last application of an inference rule in $\pi$ and define $\mathcal{E}^\ast$ as follows: 


\begin{gather*}
\AxiomC{$\pi_1$}
\noLine
\UnaryInfC{$\Gamma , B \Rightarrow  \Delta$}
\AxiomC{$\pi_2$}
\noLine
\UnaryInfC{$\Gamma \Rightarrow  A, \Delta$}
\LeftLabel{$\mathsf{\rightarrow_L}$}
\BinaryInfC{$\Gamma , A \rightarrow B \Rightarrow  \Delta$}
\DisplayProof 
\longmapsto
\AxiomC{$\mathcal{E}^\ast(\pi_1)$}
\noLine
\UnaryInfC{$\Delta, A $}
\AxiomC{$\mathcal{E}^\ast(\pi_2)$}
\noLine
\UnaryInfC{$\Gamma \Rightarrow  A, \Delta$}
\LeftLabel{$\mathsf{\rightarrow_L}$}
\RightLabel{ ,}
\BinaryInfC{$\Gamma , A \rightarrow B \Rightarrow  \Delta$}
\DisplayProof 
\end{gather*}
\begin{gather*}
\AxiomC{$\pi_0$}
\noLine
\UnaryInfC{$\Gamma, A \Rightarrow  B , \Delta$}
\LeftLabel{$\mathsf{\rightarrow_R}$}
\UnaryInfC{$\Gamma \Rightarrow  A \rightarrow B , \Delta$}
\DisplayProof 
\longmapsto
\AxiomC{$\mathcal{E}^\ast(\pi_0)$}
\noLine
\UnaryInfC{$\Gamma, A \Rightarrow  B , \Delta$}
\LeftLabel{$\mathsf{\rightarrow_R}$}
\RightLabel{ ,}
\UnaryInfC{$\Gamma \Rightarrow  A \rightarrow B , \Delta$}
\DisplayProof 
\end{gather*}
\begin{gather*}
\AxiomC{$\pi_0$}
\noLine
\UnaryInfC{$\Gamma, A, \Box A \Rightarrow   \Delta$}
\LeftLabel{$\mathsf{refl}$}
\UnaryInfC{$\Gamma , \Box A\Rightarrow  \Delta$}
\DisplayProof 
\longmapsto
\AxiomC{$\mathcal{E}^\ast(\pi_0)$}
\noLine
\UnaryInfC{$\Gamma, A, \Box A \Rightarrow   \Delta$}
\LeftLabel{$\mathsf{refl}$}
\RightLabel{ ,}
\UnaryInfC{$\Gamma , \Box A\Rightarrow  \Delta$}
\DisplayProof 
\end{gather*}
\begin{gather*}
\AxiomC{$\pi_1$}
\noLine
\UnaryInfC{$\Gamma, \Box \Pi \Rightarrow A, \Delta$}
\AxiomC{$\pi_2$}
\noLine
\UnaryInfC{$\Box \Pi \Rightarrow A$}
\LeftLabel{$\mathsf{\Box}$}
\BinaryInfC{$\Gamma, \Box \Pi \Rightarrow \Box A, \Delta$}
\DisplayProof 
\longmapsto
\AxiomC{$\mathcal{E}^\ast(\pi_1)$}
\noLine
\UnaryInfC{$\Gamma, \Box \Pi \Rightarrow A, \Delta$}
\AxiomC{$\pi_2$}
\noLine
\UnaryInfC{$\Box \Pi \Rightarrow A$}
\LeftLabel{$\mathsf{\Box}$}
\RightLabel{ ,}
\BinaryInfC{$\Gamma, \Box \Pi \Rightarrow \Box A, \Delta$}
\DisplayProof 
\end{gather*}
\begin{gather*}
\AxiomC{$\pi_1$}
\noLine
\UnaryInfC{$\Gamma  \Rightarrow  \Delta, A$}
\AxiomC{$\pi_2$}
\noLine
\UnaryInfC{$A, \Gamma \Rightarrow   \Delta$}
\LeftLabel{$\mathsf{cut}$}
\BinaryInfC{$\Gamma \Rightarrow  \Delta$}
\DisplayProof 
\longmapsto
\:\mathcal{R}_A (\mathcal{E}^\ast(\pi_1),\mathcal{E}^\ast(\pi_2)) \;.
\end{gather*}

Clearly, the mapping $\mathcal{E}^\ast$ is root-preserving, and $\mathcal{E}^\ast(\mathcal{P})\subset
\mathcal{P}_1$. We also see that $\mathcal{E}^\ast$ is nonexpansive, i.e. for any $n\in \mathbb{N}$ and any $\pi, \tau \in \mathcal P$ 
$$\pi \sim_n \tau \Rightarrow \mathcal{E}^\ast(\pi) \sim_n \mathcal{E}^\ast(\tau)\;.  $$
\end{proof}

Now we define a contractive operator $\mathcal F\colon \mathcal{N}\to \mathcal{N}$. The required cut-elimination mapping will be obtained as the fixed-point of $\mathcal{F}$. 

For a root-preserving nonexpansive mapping $\mathcal{U}$ and an $\infty$-proof $\pi$ of a sequent $\Gamma \Rightarrow \Delta$, we define $\mathcal{F} (\mathcal U)(\pi)$. 
In the case $\pi \in \mathcal{P}_1$, $\mathcal{F} (\mathcal U)(\pi)$ is introduced by induction on $\lvert \pi \rvert$.
If $\lvert\pi\rvert=0$, then we put $\mathcal{F} (\mathcal U)(\pi) =\pi$. Otherwise, consider the last application of an inference rule in $\pi$ and define $\mathcal{F} (\mathcal U)$ as follows: 

\begin{gather*}
\AxiomC{$\pi_1$}
\noLine
\UnaryInfC{$\Gamma , B \Rightarrow  \Delta$}
\AxiomC{$\pi_2$}
\noLine
\UnaryInfC{$\Gamma \Rightarrow  A, \Delta$}
\LeftLabel{$\mathsf{\rightarrow_L}$}
\BinaryInfC{$\Gamma , A \rightarrow B \Rightarrow  \Delta$}
\DisplayProof 
\longmapsto
\AxiomC{$\mathcal{F} (\mathcal U)(\pi_1)$}
\noLine
\UnaryInfC{$\Delta, A $}
\AxiomC{$\mathcal{F} (\mathcal U)(\pi_2)$}
\noLine
\UnaryInfC{$\Gamma \Rightarrow  A, \Delta$}
\LeftLabel{$\mathsf{\rightarrow_L}$}
\RightLabel{ ,}
\BinaryInfC{$\Gamma , A \rightarrow B \Rightarrow  \Delta$}
\DisplayProof 
\end{gather*}
\begin{gather*}
\AxiomC{$\pi_0$}
\noLine
\UnaryInfC{$\Gamma, A \Rightarrow  B , \Delta$}
\LeftLabel{$\mathsf{\rightarrow_R}$}
\UnaryInfC{$\Gamma \Rightarrow  A \rightarrow B , \Delta$}
\DisplayProof 
\longmapsto
\AxiomC{$\mathcal{F} (\mathcal U)(\pi_0)$}
\noLine
\UnaryInfC{$\Gamma, A \Rightarrow  B , \Delta$}
\LeftLabel{$\mathsf{\rightarrow_R}$}
\RightLabel{ ,}
\UnaryInfC{$\Gamma \Rightarrow  A \rightarrow B , \Delta$}
\DisplayProof 
\end{gather*}
\begin{gather*}
\AxiomC{$\pi_0$}
\noLine
\UnaryInfC{$\Gamma, A, \Box A \Rightarrow   \Delta$}
\LeftLabel{$\mathsf{refl}$}
\UnaryInfC{$\Gamma , \Box A\Rightarrow  \Delta$}
\DisplayProof 
\longmapsto
\AxiomC{$\mathcal{F} (\mathcal U)(\pi_0)$}
\noLine
\UnaryInfC{$\Gamma, A, \Box A \Rightarrow   \Delta$}
\LeftLabel{$\mathsf{refl}$}
\RightLabel{ ,}
\UnaryInfC{$\Gamma , \Box A\Rightarrow  \Delta$}
\DisplayProof 
\end{gather*}
\begin{gather*}
\AxiomC{$\pi_1$}
\noLine
\UnaryInfC{$\Gamma, \Box \Pi \Rightarrow A, \Delta$}
\AxiomC{$\pi_2$}
\noLine
\UnaryInfC{$\Box \Pi \Rightarrow A$}
\LeftLabel{$\mathsf{\Box}$}
\BinaryInfC{$\Gamma, \Box \Pi \Rightarrow \Box A, \Delta$}
\DisplayProof 
\longmapsto
\AxiomC{$\mathcal{F} (\mathcal U)(\pi_1)$}
\noLine
\UnaryInfC{$\Gamma, \Box \Pi \Rightarrow A, \Delta$}
\AxiomC{$\mathcal{U}(\pi_2)$}
\noLine
\UnaryInfC{$\Box \Pi \Rightarrow A$}
\LeftLabel{$\mathsf{\Box}$}
\RightLabel{ .}
\BinaryInfC{$\Gamma, \Box \Pi \Rightarrow \Box A, \Delta$}
\DisplayProof 
\end{gather*}
The mapping $\mathcal{F} (\mathcal U)$ is well defined on the set $\mathcal{P}_1$. If $\pi \notin \mathcal{P}_1$, then we put $\mathcal{F} (\mathcal U) (\pi):= \mathcal{F} (\mathcal U)(\mathcal{E}^\ast (\pi))$.

It can easily be checked that $\mathcal{F} (\mathcal U) $ is a root-preserving nonexpansive mapping.

\begin{lem}
We have that $d(\mathcal{F}(\mathcal{U}) , \mathcal{F}(\mathcal{V}))\leqslant \frac{1}{2}\cdot d(\mathcal{U}, \mathcal{V})$ for any mappings $\mathcal{U}, \mathcal{V} \in \mathcal{N}$.
\end{lem}
\begin{proof}

Let us write $\mathcal{U}\sim_n \mathcal{V}$ if $\mathcal{U}(\pi) \sim_n \mathcal{V}(\pi) $ for any $\pi \in \mathcal{P}$. We claim that for any $n \in \mathbb{N}$
$$\mathcal{U} \sim_n \mathcal{V} \Rightarrow \mathcal{F}(\mathcal{U}) \sim_{n+1} \mathcal{F}(\mathcal{V})\:.$$
Assume we have an $\infty$-proof $\pi$ and $\mathcal{U} \sim_n \mathcal{V}$. Now it can be easily proved by induction on $\lvert \pi \rvert$ that $\mathcal{F}(\mathcal{U})(\pi) \sim_{n+1} \mathcal{F}(\mathcal{V})$.

Further, we see that $\mathcal{U}\sim_n \mathcal{V}$ if and only if $d(\mathcal{U},\mathcal{V})\leqslant 2^{-n}$. 
Thus, the condition 
$$\forall n\: (\mathcal{U} \sim_n \mathcal{V} \Rightarrow \mathcal{F}(\mathcal{U}) \sim_{n+1} \mathcal{F}(\mathcal{V}))$$
is equivalent to $d(\mathcal{F}(\mathcal{U}) , \mathcal{F}(\mathcal{V}))\leqslant \frac{1}{2}\cdot d(\mathcal{U}, \mathcal{V})$. 
\end{proof}

\begin{lem}\label{moving the cut}
If $\mathcal{U}\in \mathcal{N}_n$, then $\mathcal{F}(\mathcal{U}) \in \mathcal{N}_{n+1} $. 
\end{lem}
\begin{proof}

Assume we have an $\infty$-proof $\pi$ and $\mathcal{U}\in \mathcal{N}_n$. We claim $\mathcal{F}(\mathcal{U})(\pi) \in \mathcal{P}_n $. 

If $\pi \in \mathcal{P}_1$, then it is not hard to prove by induction on $\lvert \pi \rvert$ that $\mathcal{F}(\mathcal{U})(\pi) \in \mathcal{P}_n $.
If $\pi \notin \mathcal{P}_1$, then $\mathcal{E}^\ast (\pi)\in \mathcal{P}_1$ by Lemma \ref{main-fragment cut}. Thus $\mathcal{F}(\mathcal{U})(\pi) = \mathcal{F}(\mathcal{U})(\mathcal{E}^\ast (\pi)) \in \mathcal{P}_n$ by the previous case.

\end{proof}

\begin{lem}
There exists a mapping $\mathcal{E} $ such that $\mathcal{E} \in \mathcal{N}_n$ for any $n\in \mathbb{N}$.
\end{lem}
\begin{proof}
We have that $\mathcal F\colon \mathcal{N} \to \mathcal{N}$ is a contractive operator. By the Banach fixed-point theorem, there exists a root-preserving nonexpansive mapping $\mathcal E$ such that $\mathcal F(\mathcal E)=\mathcal E$. Trivially, $\mathcal{E} \in \mathcal{N}_0=\mathcal{N}$. Hence $\mathcal{E}$ belongs to the intersection of all $\mathcal{N}_n$ for $n\in \mathbb{N}$ by Lemma \ref{moving the cut}.

\end{proof}

\begin{thm}[cut-elimination]
\label{infcuttoinf}
If $\mathsf{Grz_\infty} +\mathsf{cut}\vdash\Gamma\Rightarrow\Delta$, then $\mathsf{Grz_\infty}\vdash\Gamma\Rightarrow\Delta$.
\end{thm}
\begin{proof}

Take an $\infty$-proof of the sequent $\Gamma\Rightarrow\Delta$ in the system $\mathsf{Grz_\infty} +\mathsf{cut}$ and apply the mapping $\mathcal E$ to it. You will get an $\infty$-proof of the same sequent in the system $\mathsf{Grz_\infty}$.
\end{proof}

\section{Ordinary and non-well-founded proofs}
\label{s5}
In this section we define two translations that connect ordinary and non-well-founded sequent calculi for $\mathsf{Grz}$. 

\begin{lem}\label{AtoA}
We have  $\mathsf{Grz}_{\infty} \vdash \Gamma,A\Rightarrow A,\Delta$ for any sequent $\Gamma \Rightarrow \Delta$ and any formula $A$.
\end{lem}
\begin{proof}
Standard induction on the structure of $A$.
\end{proof}

\begin{lem}\label{Grz-schema}
We have $\mathsf{Grz}_{\infty}\vdash\Box(\Box(A \rightarrow \Box A) \rightarrow A) \Rightarrow A$ for any formula $A$.
\end{lem}
\begin{proof}
Consider an example of $\infty$--proof for the sequent $\Box(\Box(p \rightarrow \Box p) \rightarrow p) \Rightarrow p$ from Section \ref{s3}. We transform this example into an $\infty$--proof for $\Box(\Box(A \rightarrow \Box A) \rightarrow A) \Rightarrow A$ by replacing $p$ with $A$ and adding required $\infty$--proofs instead of initial sequents using Lemma \ref{AtoA}.  
\end{proof}
\begin{thm}\label{seqtoinfcut}
If $\mathsf{Grz_{Seq}}+\mathsf{cut}\vdash\Gamma\Rightarrow\Delta$, then $\mathsf{Grz}_{\infty}+\mathsf{cut}\vdash\Gamma\Rightarrow\Delta$.
\end{thm}
\begin{proof}
Assume $\pi$ is a proof of $\Gamma\Rightarrow\Delta$ in $\mathsf{Grz_{Seq}}+\mathsf{cut}$. By induction on the size of $\pi$ we prove $\mathsf{Grz}_{\infty}+\mathsf{cut}\vdash\Gamma\Rightarrow\Delta$. 

If $\Gamma \Rightarrow \Delta $ is an initial sequent of $\mathsf{Grz_{Seq}}+\mathsf{cut}$, then it is provable in $\mathsf{Grz}_{\infty}+\mathsf{cut}$ by Lemma \ref{AtoA}.
Otherwise, consider the last application of an inference rule in $\pi$. 

The only non-trivial case is when the proof $\pi$ has the form 
\begin{gather*}
\AxiomC{$\pi^\prime$}
\noLine
\UnaryInfC{$\Box \Pi,\Box(A\to\Box A)\Rightarrow A$}
\LeftLabel{$\mathsf{\Box_{Grz}}$}
\RightLabel{ ,}
\UnaryInfC{$\Sigma,\Box\Pi\Rightarrow \Box A, \Lambda$}
\DisplayProof
\end{gather*} 
where $\Sigma,\Box\Pi = \Gamma$ and $\Box A, \Lambda = \Delta$. By the induction hypothesis there is an $\infty$--proof $\xi$ of $\Box \Pi,\Box(A\to\Box A)\Rightarrow A$ in $\mathsf{Grz}_{\infty}+\mathsf{cut}$.

We have the following $\infty$--proof $\lambda$ of $\Box \Pi\Rightarrow A$ in $\mathsf{Grz}_{\infty}+\mathsf{cut} $:
\begin{gather*}
\AxiomC{$\xi^{\prime}$}
\noLine
\UnaryInfC{$\Box \Pi,\Box(A\to\Box A)\Rightarrow A,A$}
\LeftLabel{$\mathsf{\to_R}$}
\UnaryInfC{$\Box \Pi\Rightarrow G,A$}
\AxiomC{$\xi$}
\noLine
\UnaryInfC{$\Box \Pi,\Box(A\to\Box A)\Rightarrow A$}
\LeftLabel{$\mathsf{\to_R}$}
\UnaryInfC{$\Box \Pi\Rightarrow G$}
\LeftLabel{$\Box$}
\BinaryInfC{$\Box \Pi\Rightarrow\Box G,A$}
\AxiomC{$\theta$}
\noLine
\UnaryInfC{$\Box\Pi,\Box G \Rightarrow A$}
\LeftLabel{$\mathsf{cut}$}
\RightLabel{ ,}
\BinaryInfC{$\Box\Pi\Rightarrow A$}
\DisplayProof
\end{gather*}
where $G= \Box(A \rightarrow \Box A) \rightarrow A$, $\xi^{\prime}$ is an $\infty$--proof of $\Box \Pi,\Box(A\to\Box A)\Rightarrow A,A$ obtained from $\xi$ by Lemma \ref{strongweakening} and $\theta$ is an $\infty$--proof of $\Box\Pi,\Box G \Rightarrow A$, which exists by Lemma \ref{Grz-schema} and Lemma \ref{strongweakening}.

The required $\infty$--proof for $\Sigma,\Box\Pi\Rightarrow \Box A, \Delta$ has the form
\begin{gather*}
\AxiomC{$\lambda^\prime$}
\noLine
\UnaryInfC{$\Sigma,\Box\Pi\Rightarrow A,\Lambda$}
\AxiomC{$\lambda$}
\noLine
\UnaryInfC{$\Box\Pi\Rightarrow A$}
\LeftLabel{$\Box$}
\RightLabel{ ,}
\BinaryInfC{$\Sigma,\Box\Pi\Rightarrow \Box A, \Lambda$}
\DisplayProof
\end{gather*}
where $\lambda^\prime$ is an $\infty$--proof for the sequent $\Gamma,\Box\Pi\Rightarrow A,\Delta$ obtained from $\lambda$ by Lemma \ref{weakening}.

The cases of other inference rules being last in $\pi$ are straightforward, so we omit them.

\end{proof}

\begin{lem}\label{weakening}
The rule
\begin{gather*}
\AxiomC{$\Gamma\Rightarrow\Delta$}
\LeftLabel{$\mathsf{weak}$}
\UnaryInfC{$\Pi,\Gamma\Rightarrow\Delta,\Sigma$}
\DisplayProof
\end{gather*}
is admissible in $\mathsf{Grz_{Seq}} $.
\end{lem}
\begin{proof}
Standard induction on the structure of a proof of $\Gamma\Rightarrow\Delta$.
\end{proof}

For a sequent $\Gamma\Rightarrow\Delta$, let $Sub(\Gamma\Rightarrow\Delta)$ be the set of all subformulas of the formulas from $\Gamma \cup\Delta$.
For  a finite set of formulas $\Lambda$, set $\Lambda^\ast:=\{\Box(A\to\Box A)\mid A\in\Lambda\}$.

\begin{lem} \label{translation}
If $\mathsf{Grz_\infty}\vdash \Gamma\Rightarrow\Delta$, then $\mathsf{Grz_{Seq}} \vdash \Lambda^\ast,\Gamma\Rightarrow\Delta$ for any finite set of formulas $\Lambda$.
\end{lem}
\begin{proof}
Assume $\pi$ is an $\infty$--proof of the sequent $\Gamma\Rightarrow\Delta$ in $\mathsf{Grz}_\infty$ and $\Lambda$ is a finite set of formulas.
By induction on the number of elements in the finite set $Sub(\Gamma\Rightarrow\Delta)\setminus \Lambda$ with a subinduction on $\lvert \pi \rvert$, we prove $\mathsf{Grz_{Seq}} \vdash \Lambda^\ast,\Gamma\Rightarrow\Delta$. 


If $\lvert \pi \rvert=0$, then $\Gamma\Rightarrow\Delta$ is an initial sequent. We see that the sequent $\Lambda^\ast,\Gamma\Rightarrow\Delta$ is an initial sequent and it is provable in $\mathsf{Grz_{Seq}}$.
Otherwise, consider the last application of an inference rule in $\pi$. 

Case 1. Suppose that $\pi$ has the form
\begin{gather*}
\AxiomC{$\pi^\prime$}
\noLine
\UnaryInfC{$\Gamma,A\Rightarrow B,\Sigma$}
\LeftLabel{$\mathsf{\to_R}$}
\RightLabel{ ,}
\UnaryInfC{$\Gamma\Rightarrow A\to B,\Sigma$}
\DisplayProof
\end{gather*}
where $A\to B,\Sigma = \Delta$.
Notice that $\lvert \pi^\prime \rvert < \lvert \pi \rvert $. By the induction hypothesis for $\pi^\prime$ and $\Lambda$, the sequent $\Lambda^\ast,\Gamma,A\Rightarrow B,\Sigma$ is provable in  $\mathsf{Grz_{Seq}}$. 
Applying the rule ($\mathsf{\to_R}$) to it, we obtain that the sequent $\Lambda^\ast,\Gamma\Rightarrow\Delta$ is provable in $\mathsf{Grz_{Seq}}$.

Case 2. Suppose that $\pi$ has the form
\begin{gather*}
\AxiomC{$\pi^\prime$}
\noLine
\UnaryInfC{$\Sigma, B\Rightarrow \Delta$}
\AxiomC{$\pi^{\prime\prime}$}
\noLine
\UnaryInfC{$\Sigma \Rightarrow A,\Delta$}
\LeftLabel{$\mathsf{\to_L}$}
\RightLabel{ ,}
\BinaryInfC{$\Sigma, A\to B\Rightarrow \Delta$}
\DisplayProof
\end{gather*}
where $\Sigma, A\to B = \Gamma$. We see that $\lvert \pi^\prime \rvert < \lvert \pi \rvert $. By the induction hypothesis for $\pi^\prime$ and $\Lambda$, the sequent $\Lambda^\ast,\Sigma, B\Rightarrow \Delta$ is provable in  $\mathsf{Grz_{Seq}}$. Analogously, we have $\mathsf{Grz_{Seq}} \vdash \Lambda^\ast,\Sigma \Rightarrow A,\Delta$. Applying the rule ($\mathsf{\to_L}$), we obtain that the sequent $\Lambda^\ast,\Sigma, A\to B \Rightarrow\Delta$ is provable in $\mathsf{Grz_{Seq}}$.

Case 3. Suppose that $\pi$ has the form
\begin{gather*}
\AxiomC{$\pi^\prime$}
\noLine
\UnaryInfC{$\Sigma,A,\Box A\Rightarrow \Delta$}
\LeftLabel{$\mathsf{refl}$}
\RightLabel{ ,}
\UnaryInfC{$\Sigma,\Box A\Rightarrow \Delta$}
\DisplayProof
\end{gather*}
where $\Sigma, \Box A = \Gamma$. We see that $\lvert \pi^\prime \rvert < \lvert \pi \rvert $. By the induction hypothesis for $\pi^\prime$ and $\Lambda$, the sequent $\Lambda^\ast,\Sigma, A, \Box A\Rightarrow \Delta$ is provable in  $\mathsf{Grz_{Seq}}$. Applying the rule ($\mathsf{refl}$), we obtain $\mathsf{Grz_{Seq}} \vdash \Lambda^\ast,\Sigma, \Box A\Rightarrow\Delta$. 

Case 4. Suppose that $\pi$ has the form
\begin{gather*}
\AxiomC{$\pi^\prime$}
\noLine
\UnaryInfC{$\Phi, \Box \Pi \Rightarrow A, \Sigma$}
\AxiomC{$\pi^{\prime\prime}$}
\noLine
\UnaryInfC{$\Box \Pi \Rightarrow A$}
\LeftLabel{$\mathsf{\Box}$}
\RightLabel{ ,}
\BinaryInfC{$\Phi, \Box \Pi \Rightarrow \Box A, \Sigma$}
\DisplayProof
\end{gather*}
where $\Phi, \Box \Pi = \Gamma$ and $\Box A, \Sigma =\Delta$.

Subcase 4.1: the formula $A$ belongs to $\Lambda$. We see that $\lvert \pi^\prime \rvert < \lvert \pi \rvert $. By the induction hypothesis for $\pi^\prime$ and $\Lambda$, the sequent $\Lambda^\ast,\Phi, \Box \Pi \Rightarrow A, \Sigma$ is provable in  $\mathsf{Grz_{Seq}}$. 
Then we see
\begin{gather*}
\AxiomC{$\mathsf{Ax}$}
\noLine
\UnaryInfC{$\Lambda^\ast,\Box A,\Phi, \Box \Pi \Rightarrow \Box A, \Sigma$}
\AxiomC{$\Lambda^\ast,\Phi, \Box \Pi \Rightarrow A,  \Sigma$}
\LeftLabel{$\mathsf{weak}$}
\UnaryInfC{$\Lambda^\ast,\Phi, \Box \Pi \Rightarrow A,  \Box A,\Sigma$}
\LeftLabel{$\mathsf{\to_L}$}
\BinaryInfC{$(\Lambda\backslash\{A\})^\ast,A\to\Box A,\Box(A\to\Box A),\Phi, \Box \Pi \Rightarrow \Box A, \Sigma$}
\LeftLabel{$\mathsf{refl}$}
\RightLabel{ ,}
\UnaryInfC{$(\Lambda\backslash\{A\})^\ast,\Box(A\to\Box A),\Phi, \Box \Pi \Rightarrow \Box A, \Sigma$}
\DisplayProof
\end{gather*}
where the rule ($\mathsf{weak}$) is admissible by Lemma \ref{weakening}.

Subcase 4.2: the formula $A$ doesn't belong to $\Lambda$. We have that the number of elements in $Sub(\Box\Pi\Rightarrow A)\setminus(\Lambda\cup \{A\})$ is strictly less than the number of elements in $Sub(\Phi, \Box \Pi \Rightarrow \Box A, \Sigma)\setminus\Lambda$. Therefore, by the induction hypothesis for $\pi^{\prime\prime}$ and $\Lambda\cup \{A\}$, the sequent $\Lambda^\ast,\Box(A\to\Box A),\Box \Pi \Rightarrow A$ is provable in  $\mathsf{Grz_{Seq}}$. Then we have
\begin{gather*}
\AxiomC{$\Lambda^\ast,\Box(A\to\Box A),\Box \Pi \Rightarrow A$}
\LeftLabel{$\mathsf{\Box_{Grz}}$}
\RightLabel{ .}
\UnaryInfC{$\Lambda^\ast,\Phi, \Box \Pi \Rightarrow \Box A, \Sigma$}
\DisplayProof
\end{gather*}
\end{proof}
From Lemma \ref{translation} we immediately obtain the following theorem.
\begin{thm}\label{inftoseq}
If $\mathsf{Grz_\infty}\vdash \Gamma\Rightarrow\Delta$, then $\mathsf{Grz_{Seq}} \vdash \Gamma\Rightarrow\Delta$.
\end{thm}

Theorem \ref{cutelimgrz} is now established as a direct consequence of Theorem \ref{seqtoinfcut}, Theorem \ref{infcuttoinf}, and Theorem \ref{inftoseq}.

\section{Conclusion and Future Work}\label{s6}

Recall that the Craig interpolation property for a logic $\mathsf{L}$ says that if $A$ implies $B$, then there is an interpolant, that is, a formula $I$ containing only common variables of $A$ and $B$ such that $A$ implies $I$ and $I$ implies $B$. The Lyndon interpolation property is a strengthening of the Craig one that also takes into consideration negative and positive occurrences of the shared propositional variables; that is, the variables occurring in $I$ positively (negatively) must also occur both in $A$ and $B$ positively (negatively).

Though the Grzegorczyk logic has the Lyndon interpolation property \cite{Maks2}, there were seemingly no syntactic proofs of this result.  
It is unclear how Lyndon interpolation can be obtained from previously introduced sequent systems for $\mathsf{Grz}$ \cite{Avron,Borga,Negri} by direct proof-theoretic arguments because these systems contain inference rules in which a polarity change occurs under the passage from the principal formula in the conclusion to its immediate ancestors in the premise. Using our system $\mathsf{Grz}_\infty$ we believe that we can obtain a syntactic proof of Lyndon interpolation for the modal Grzegorczyk logic as an application of our cut-elimination theorem.

We also believe that every provable $\mathsf{Grz}_\infty$ sequent has a proof that is a regular tree (has only finite amout of distinct subtrees). This gives a possibility of proof system for the logic $\mathsf{Grz}$ with cyclical proofs, like the system introduced in \cite{Sham}.

\section{Acknoledgements}
The article was prepared within the framework of the Basic Research Program at the National Research University Higher School of Economics (HSE) and supported within the framework of a subsidy by the Russian Academic Excellence Project '5-100'. Both authors also acknowledge support from the Russian Foundation for Basic Research (grant no. 15-01-09218a).

\newpage
\section*{Appendix.}
\subsection*{Proof of Lemma \ref{smallcut}}
Assume we have two $\infty$-proofs $\pi^\prime$ and $\pi^{\prime\prime}$ from $\mathcal P_1$. If there is no application of the cut rule to these $\infty$-proofs with the cut formula $p$, then we put $\mathcal R_{p}(\pi^\prime,\pi^{\prime\prime}) := \pi^\prime$. In the converse case, there is a sequent $\Gamma\Rightarrow \Delta$ such that $\pi^\prime$ is an $\infty$-proof of $\Gamma\Rightarrow \Delta, p$ and $\pi^{\prime\prime}$ is an $\infty$-proof of $p, \Gamma\Rightarrow \Delta$.
We define $\mathcal R_{p}(\pi^\prime,\pi^{\prime\prime})$ by induction on $\lvert \pi^\prime \rvert$.

If $\lvert \pi^\prime\rvert=0$, then $\Gamma\Rightarrow \Delta, p$ is an initial sequent. Suppose that $\Gamma\Rightarrow \Delta$ is also an initial sequent. Then $\mathcal R_{p}(\pi^\prime,\pi^{\prime\prime})$ is defined as the $\infty$-proof consisting only of this initial sequent. Otherwise, $\Gamma$ has the form $p,\Phi$, and $\pi^{\prime\prime}$ is an $\infty$-proof of $p,p,\Phi \Rightarrow \Delta$. Applying the nonexpansive mapping $\mathit{acl}_p$ from Lemma \ref{weakcontraction}, we put $\mathcal R_{p}(\pi^\prime,\pi^{\prime\prime}) := \mathit{acl}_p (\pi^{\prime\prime})$.   

Now suppose that $\lvert \pi^\prime \rvert >0$. We consider the last application of an inference rule in $\pi^\prime$. 

Case 1. The $\infty$-proof $\pi^\prime$ has the form
\begin{gather*}
\AxiomC{$\pi^\prime_0$}
\noLine
\UnaryInfC{$\Gamma,A\Rightarrow B,\Sigma,p$}
\LeftLabel{$\mathsf{\to_R}$}
\RightLabel{ ,}
\UnaryInfC{$\Gamma\Rightarrow A\to B,\Sigma,p$}
\DisplayProof
\end{gather*}
where $A\to B,\Sigma = \Delta$.
Notice that $\lvert \pi^\prime_0 \rvert < \lvert \pi^\prime \rvert $. In addition, $\pi^{\prime\prime}$ is an $\infty$-proof of $p,\Gamma\Rightarrow A\to B,\Sigma$. We define $\mathcal R_{p}(\pi^\prime,\pi^{\prime\prime}) $ as
\begin{gather*}
\AxiomC{$\mathcal{R}_p(\pi^\prime_0, \mathit{i}_{A \to B}(\pi^{\prime\prime}))$}
\noLine
\UnaryInfC{$\Gamma,A\Rightarrow B,\Sigma$}
\LeftLabel{$\mathsf{\to_R}$}
\RightLabel{ ,}
\UnaryInfC{$\Gamma\Rightarrow A\to B,\Sigma$}
\DisplayProof
\end{gather*}
where $\mathit{i}_{A\to B}$ is a nonexpansive mapping from Lemma \ref{inversion}.

Case 2. The $\infty$-proof $\pi^\prime$ has the form
\begin{gather*}
\AxiomC{$\pi^\prime_0$}
\noLine
\UnaryInfC{$\Sigma, B\Rightarrow \Delta, p$}
\AxiomC{$\pi^{\prime}_1$}
\noLine
\UnaryInfC{$\Sigma \Rightarrow A,\Delta, p$}
\LeftLabel{$\mathsf{\to_L}$}
\RightLabel{ ,}
\BinaryInfC{$\Sigma, A\to B\Rightarrow \Delta, p$}
\DisplayProof
\end{gather*}
where $\Sigma, A\to B = \Gamma$. We see that $\lvert \pi^\prime_0 \rvert < \lvert \pi^\prime \rvert $ and $\lvert \pi^\prime_1 \rvert < \lvert \pi^\prime \rvert $. Also, 
$\pi^{\prime\prime}$ is an $\infty$-proof of $p,\Sigma, A\to B\Rightarrow \Delta$. We define $\mathcal R_{p}(\pi^\prime,\pi^{\prime\prime}) $ as
\begin{gather*}
\AxiomC{$\mathcal{R}_p (\pi^\prime_0, \mathit{li}_{A\to B} (\pi^{\prime\prime}))$}
\noLine
\UnaryInfC{$\Sigma, B\Rightarrow \Delta, p$}
\AxiomC{$\mathcal{R}_p (\pi^{\prime}_1, \mathit{ri}_{A\to B} (\pi^{\prime\prime}))$}
\noLine
\UnaryInfC{$\Sigma \Rightarrow A,\Delta, p$}
\LeftLabel{$\mathsf{\to_L}$}
\RightLabel{ ,}
\BinaryInfC{$\Sigma, A\to B\Rightarrow \Delta, p$}
\DisplayProof
\end{gather*}
where $\mathit{li}_{A\to B}$ and $\mathit{ri}_{A\to B}$ are nonexpansive mappings from Lemma \ref{inversion}.

Case 3. The $\infty$-proof $\pi^\prime$ has the form
\begin{gather*}
\AxiomC{$\pi^\prime_0$}
\noLine
\UnaryInfC{$\Sigma,A,\Box A\Rightarrow \Delta,p$}
\LeftLabel{$\mathsf{refl}$}
\RightLabel{ ,}
\UnaryInfC{$\Sigma,\Box A\Rightarrow \Delta,p$}
\DisplayProof
\end{gather*}
where $\Sigma, \Box A = \Gamma$. We have that $\lvert \pi^\prime \rvert < \lvert \pi \rvert $. 
Define $\mathcal R_{p}(\pi^\prime,\pi^{\prime\prime}) $ as
\begin{gather*}
\AxiomC{$\mathcal{R}_p(\pi^\prime_0, \mathit{wk}_{ A, \emptyset} (\pi^{\prime\prime})$}
\noLine
\UnaryInfC{$\Sigma,A,\Box A\Rightarrow \Delta$}
\LeftLabel{$\mathsf{refl}$}
\RightLabel{ ,}
\UnaryInfC{$\Sigma,\Box A\Rightarrow \Delta$}
\DisplayProof
\end{gather*}
where $\mathit{wk}_{A, \emptyset}$ is the nonexpansive mapping from Lemma \ref{strongweakening}.

Case 4. Now consider the final case when $\pi^\prime$ has the form
\begin{gather*}
\AxiomC{$\pi^\prime_0$}
\noLine
\UnaryInfC{$\Phi, \Box \Pi \Rightarrow A, \Sigma, p$}
\AxiomC{$\pi^{\prime}_1$}
\noLine
\UnaryInfC{$\Box \Pi \Rightarrow A$}
\LeftLabel{$\mathsf{\Box}$}
\RightLabel{ ,}
\BinaryInfC{$\Phi, \Box \Pi \Rightarrow \Box A, \Sigma, p$}
\DisplayProof
\end{gather*}
where $\Phi, \Box \Pi = \Gamma$ and $\Box A, \Sigma =\Delta$. Notice that $\lvert \pi^\prime_0 \rvert < \lvert \pi^\prime \rvert $. In addition, $\pi^{\prime\prime}$ is an $\infty$-proof of $p,\Phi, \Box \Pi \Rightarrow \Box A, \Sigma$. We define $\mathcal R_{p}(\pi^\prime,\pi^{\prime\prime}) $ as
\begin{gather*}
\AxiomC{$\mathcal{R}_p(\pi^\prime_0, \mathit{li}_{\: \Box A}(\pi^{\prime\prime}))$}
\noLine
\UnaryInfC{$\Phi, \Box \Pi \Rightarrow A, \Sigma$}
\AxiomC{$\pi^{\prime}_1$}
\noLine
\UnaryInfC{$\Box \Pi \Rightarrow A$}
\LeftLabel{$\mathsf{\Box}$}
\RightLabel{ ,}
\BinaryInfC{$\Phi, \Box \Pi \Rightarrow \Box A, \Sigma$}
\DisplayProof
\end{gather*}
where $\mathit{li}_{\: \Box A}$ is a nonexpansive mapping from Lemma \ref{inversion}.

The mapping $\mathcal{R}_p$ is well defined. It remains
to check that $\mathcal{R}_p$ is nonexpansive, i.e. for any $n\in \mathbb{N}$ and any $\pi^\prime$, $\pi^{\prime\prime}$, $\tau^\prime$, $\tau^{\prime\prime}$ from $ \mathcal P_0$ 
$$(\pi^\prime \sim_n \tau^\prime \wedge \pi^{\prime\prime} \sim_n \tau^{\prime\prime}) \Rightarrow \mathcal{R}_p(\pi^\prime, \pi^{\prime\prime}) \sim_n \mathcal{R}_p(\tau^\prime, \tau^{\prime\prime})\;.  $$
This condition is checked by structural induction 
on the inductively defined relation $\pi^\prime \sim_n \tau^\prime$ 
in a straightforward way. So we omit further details.

\subsection*{Proof of Lemma \ref{boxcut}}

Assume we have two $\infty$-proofs $\pi^\prime$ and $\pi^{\prime\prime}$ from $\mathcal P_1$. If there is no application of the cut rule to these $\infty$-proofs with the cut formula $\Box B$, then we put $\mathcal{R}_{\Box B}(\pi^\prime,\pi^{\prime\prime}) := \pi^\prime$. In the converse case, we define $\mathcal{R}_{\Box B}(\pi^\prime,\pi^{\prime\prime})$ by induction on $\lvert \pi^\prime \rvert + \lvert \pi^{\prime\prime} \rvert$.

If $\lvert \pi^\prime\rvert=0$ or $\lvert \pi^{\prime\prime} \rvert=0$, then $\Gamma\Rightarrow \Delta$ is an initial sequent. Then $\mathcal R_{\Box B}(\pi^\prime,\pi^{\prime\prime})$ is defined as the $\infty$-proof consisting only of this initial sequent. 

Now suppose that $\lvert \pi^\prime \rvert >0$. We consider the last application of an inference rule in $\pi^\prime$. If the principal formula of this inference is not $\Box B$, then $\mathcal R_{\Box B}(\pi^\prime,\pi^{\prime\prime}) $ is defined similarly to the four cases of Lemma \ref{smallcut}. 

We can now assume that $\pi^\prime$ has the form
\begin{gather*}
\AxiomC{$\pi^\prime_0$}
\noLine
\UnaryInfC{$\Phi, \Box \Pi \Rightarrow B, \Sigma$}
\AxiomC{$\pi^{\prime}_1$}
\noLine
\UnaryInfC{$\Box \Pi \Rightarrow B$}
\LeftLabel{$\mathsf{\Box}$}
\RightLabel{ ,}
\BinaryInfC{$\Phi, \Box \Pi \Rightarrow \Box B, \Sigma$}
\DisplayProof
\end{gather*}

Consider the last application of an inference rule in $\pi^{\prime\prime}$. If the rule used was $\mathsf{\to_L}$, $\mathsf{\to_R}$, $\mathsf{refl}$ with the principal formula being not $\Box B$, or the rule $\mathsf{\Box}$ without the formula $\Box B$ in the right premise, then $\mathcal R_{\Box B}(\pi^\prime,\pi^{\prime\prime})$ can also be defined similarly to the previous case.

Otherwise, we have the following cases. 

Case A. The $\infty$-proof $\pi^{\prime\prime}$ has the form
\begin{gather*}
\AxiomC{$\pi^{\prime\prime}_0$}
\noLine
\UnaryInfC{$\Gamma,B,\Box B\Rightarrow \Delta$}
\LeftLabel{$\mathsf{refl}$}
\RightLabel{ .}
\UnaryInfC{$\Gamma,\Box B\Rightarrow \Delta$}
\DisplayProof
\end{gather*}

Since that $\lvert \pi^{\prime\prime}_0 \rvert < \lvert \pi^{\prime\prime} \rvert $, we can define $\mathcal R_{\Box B}(\pi^\prime,\pi^{\prime\prime}) $ as
$$\mathcal R_{B}(\pi^\prime_0,\mathcal R_{\Box B}(\pi^\prime,\pi^{\prime\prime}_0)).$$

Case B. The $\infty$-proof $\pi^{\prime\prime}$ has the form
\begin{gather*}
\AxiomC{$\pi^{\prime\prime}_0$}
\noLine
\UnaryInfC{$\Phi^\prime, \Box B, \Box \Pi^\prime \Rightarrow C, \Sigma^\prime$}
\AxiomC{$\pi^{\prime\prime}_1$}
\noLine
\UnaryInfC{$\Box B,\Box \Pi^\prime \Rightarrow C$}
\LeftLabel{$\mathsf{\Box}$}
\RightLabel{ ,}
\BinaryInfC{$\Phi^\prime, \Box B, \Box \Pi^\prime \Rightarrow \Box C, \Sigma^\prime$}
\DisplayProof
\end{gather*}

Since $\lvert \pi^{\prime\prime}_0 \rvert < \lvert \pi^{\prime\prime} \rvert $ and the sequents $\Phi^\prime, \Box \Pi^\prime \Rightarrow \Box C, \Sigma^\prime$ and $\Gamma\Rightarrow\Delta$ are equal, we can define $\mathcal R_{\Box B}(\pi^\prime,\pi^{\prime\prime}) $ as
\begin{gather*}
\AxiomC{$\mathcal R_{\Box B}(\pi',\pi''_0)$}
\noLine
\UnaryInfC{$\Phi^\prime, \Box \Pi^\prime \Rightarrow C, \Sigma^\prime$}
\AxiomC{$\mathit{wk_{\Box\Pi^\prime\backslash\Box\Pi,C}}(\pi^\prime_1)$}
\noLine
\UnaryInfC{$\Box\Pi\cup\Box\Pi^\prime\Rightarrow B,C$}
\AxiomC{$\mathit{wk_{\Box\Pi^\prime\backslash\Box\Pi,\varnothing}}(\pi^\prime_1)$}
\noLine
\UnaryInfC{$\Box\Pi\cup\Box\Pi^\prime\Rightarrow B$}
\LeftLabel{$\mathsf{\Box}$}
\BinaryInfC{$\Box\Pi\cup\Box\Pi^\prime\Rightarrow \Box B,C$}
\AxiomC{$\mathit{wk_{\Box\Pi\backslash\Box\Pi^\prime,\varnothing}}(\pi^{\prime\prime}_1)$}
\noLine
\UnaryInfC{$\Box\Pi\cup\Box\Pi^\prime,\Box B\Rightarrow C$}
\LeftLabel{$\mathsf{cut}$}
\BinaryInfC{$\Box\Pi\cup\Box\Pi^\prime\Rightarrow C$}
\LeftLabel{$\mathsf{\Box}$}
\BinaryInfC{$\Phi^\prime, \Box \Pi^\prime \Rightarrow \Box C, \Sigma^\prime$}
\DisplayProof
\end{gather*} 

where $\mathit{wk}_{-,-}$ is a nonexpansive mapping from Lemma \ref{strongweakening}. Since the instance of the rule $\mathsf{cut}$ is not in the main fragment, this proof is in $\mathcal P_1$.

\end{document}